\documentclass[12pt]{amsart}
\usepackage{epsfig,color}
\usepackage{blindtext}
\usepackage{graphicx}
\usepackage{enumitem}
\usepackage{url}
\usepackage{amssymb}
\usepackage{graphicx,import}
\usepackage{comment}
\usepackage{esint}
\usepackage{xcolor}
\usepackage{mathtools}
\usepackage{comment}
\usepackage{tikz}

\usepackage[margin = 1in] {geometry}

\usepackage{hyperref}
\usepackage{dsfont}

\newtheorem{theorem}{Theorem}[section]

\newtheorem{proposition}[theorem]{Proposition}
\newtheorem{lemma}[theorem]{Lemma}

\theoremstyle{definition}

\newtheorem*{remark*}{Remark}
\newtheorem*{acknowledgements}{Acknowledgements}
\newtheorem{remark}[theorem]{Remark}

\numberwithin{equation}{section}

\newcommand{\R}{\mathbb{R}}
\newcommand{\C}{\mathbb{C}}

\newcommand{\Z}{\mathbb{Z}}
\newcommand{\eps}{\varepsilon}

\title{A spinor proof of the classification of stable minimal surfaces in $\mathbb{R}^3$}
\author{Douglas Stryker}
\address{Department of Mathematics, Stanford University, Building 380, Stanford, CA 94305, USA}
\email{dstryker@stanford.edu}

\begin{document}

\begin{abstract}
We give a proof that every complete two--sided stable minimal surface in $\mathbb{R}^3$ is flat using the index theory for Dirac operators on twisted spinor bundles.
\end{abstract}

\maketitle

\section{Introduction}

We use the index theory for Dirac operators on twisted spinor bundles over noncompact manifolds introduced by \cite{Gromov-Lawson:spin} to give a new proof of the classification of complete two--sided stable minimal surfaces in $\R^3$:
\begin{theorem}[\cite{doCarmo-Peng:stable, Fischer-Colbrie-Schoen:stable, Pogorelov:stable}]\label{thm:main}
Every complete two--sided stable minimal immersion $\Sigma^2 \to \R^3$ is flat.
\end{theorem}
Each of the original proofs \cite{doCarmo-Peng:stable, Fischer-Colbrie-Schoen:stable, Pogorelov:stable}, as well as the subsequent proofs \cite{Colding-Minicozzi:stable, Munteanu-Sung-Wang:stable}, employs either the Gauss--Bonnet formula or the uniformization theorem. In our spinor proof, the only dimension--specific facts that we exploit are:
\begin{enumerate}
	\item minimal surfaces in $\R^3$ have nonpositive sectional curvature (whereas in higher dimensions, minimal hypersurfaces in $\R^{n+1}$ have nonpositive \emph{Ricci} curvature),
	\item a circle with large length admits a degree one map to the unit circle with small Lipschitz constant (whereas in higher dimensions, a closed connected Riemannian manifold with large volume does not necessarily admit a degree one map to the unit sphere with small Lipschitz constant, e.g.\ a long skinny ellipsoid).
\end{enumerate}

We emphasize that, by using the spinor technique, this proof is fundamentally about \emph{scalar curvature} (not, for instance, Ricci curvature or bi--Ricci curvature); consequently, there seems to be no chance to classify stable minimal hypersurfaces in $\R^{n+1}$ for $n \geq 4$ along these lines (see the discussion in \cite[\S1.1]{Chodosh-Li-Minter-Stryker:stable}). For progress on the classification in higher dimensions, we refer the reader to \cite{Chodosh-Li:r4, Chodosh-Li:r4aniso, Catino-Mastrolia-Roncoroni:r4, Chodosh-Li-Minter-Stryker:stable, Antonelli-Xu:spec, Mazet:r6, CCMMR}. We also emphasize that our proof does not use geometric measure theory, unlike many of the arguments in higher dimensions (with the exception of \cite{Chodosh-Li:r4, Catino-Mastrolia-Roncoroni:r4, CCMMR}). 

We recall that by the log--cutoff (see \cite[Theorem 2.10]{Colding-Minicozzi:minimal}), it suffices to show that $\Sigma$ has quadratic area growth. Our proof is divided into three parts.

\subsection{Twisted spinors and spectral scalar curvature}
The main ingredient from the spinor theory is the following obstruction to positive ``spectral'' scalar curvature.
\begin{theorem}\label{thm:main-obstruction}
There are constants $C_n > 0$ so that the following holds. Suppose $(M^n, g)$ is a complete spin manifold of dimension $n \geq 2$, $0 < \alpha < 4$ is a constant, and $h : M \to \R_{>0}$ is a smooth positive function satisfying
\begin{equation}\label{eqn:main-spec} \int_M \alpha |\nabla \phi|^2 + (R_M - h)\phi^2 \geq 0 \end{equation}
for all smooth, compactly supported functions $\phi$ on $M$, where $R_M$ is the scalar curvature.

Then there is no Lipschitz map $f: M \to \mathbb{S}^n$ satisfying
\begin{itemize}
	\item $f$ is has nonzero degree,
	\item $f$ is constant outside a compact subset of $M$, and
	\item $f$ is $C_n h$--contracting on 2--forms (meaning $|df_x(v_1) \wedge df_x(v_2)| \leq C_nh(x)|v_1 \wedge v_2|$ for any $x \in M$ and $v_1,\ v_2$ vectors in $T_xM$).
\end{itemize}
\end{theorem}

We expect that Theorem \ref{thm:main-obstruction} is known, or at least unsurprising, to experts (see the related works \cite{Bei-Pipoli:spin, Hirsch-Kazaras-Khuri-Zhang:spec, Chai-Pyo-Wan:spec}). We note that when $\alpha = 0$, the curvature assumption \eqref{eqn:main-spec} reduces to the assumption of pointwise positive scalar curvature, where the result follows from \cite{Gromov-Lawson:spin}. One can view the relationship between our Theorem \ref{thm:main-obstruction} (also \cite[Theorem 0.1]{Bei-Pipoli:spin}) and the pointwise positive scalar curvature results of \cite{Gromov-Lawson:compact, Gromov-Lawson:spin} by the following analogy with the study of harmonic 1--forms. By the Bochner formula, any complete manifold with positive Ricci curvature admits no $L^2$--harmonic 1--forms. The same result also hold for two--sided stable minimal hypersurfaces in positively curved ambient spaces (see for instance \cite{Palmer:1-forms, Miyaoka:1-forms, Cao-Shen-Zhu:ends}), which fundamental exploits the fact that these hypersurfaces have positive ``spectral'' Ricci curvature.

There are two ways to approach the proof of Theorem \ref{thm:main-obstruction}. The first (and more classical) approach is to look at the warped product
\[ (M \times S^1,\ g\oplus C u^2g_{\mathbb{S}^1}), \]
where $u: M \to \R$ is a positive function satisfying $-\triangle_M u + (R_M - h)u = 0$, and $C > 0$ is any positive constant. As demonstrated by \cite[Proof of Theorem 4]{Fischer-Colbrie-Schoen:stable}, this warped product has pointwise positive scalar curvature. By taking $C$ sufficiently large, we can construct a new suitable map using the map $f$ and a smashing map (see \cite[Proof of Theorem 6.12]{Gromov-Lawson:spin}, and, for example, Figure \ref{fig:smash}), and then apply the result of \cite{Gromov-Lawson:spin} for pointwise positive scalar curvature. For example, this approach is taken in \cite[Theorem 4.1]{Orikasa:spin}.

We offer an alternative (and more direct) approach to the proof of Theorem \ref{thm:main-obstruction}. The condition \eqref{eqn:main-spec} is sufficient to directly deduce the nonexistence of a nontrivial $L^2$ section of the twisted spinor bundle constructed using the map $f$.\footnote{In fact, in the untwisted case, this was already shown by \cite{Bei-Pipoli:spin}. For similar computations in the twisted context, see \cite{Hirsch-Kazaras-Khuri-Zhang:spec, Chai-Pyo-Wan:spec}.} There is, however, a small subtlety in adapting the relative index theory to this setting. Indeed, in \cite{Gromov-Lawson:spin}, the index theory is developed under the assumption of pointwise lower bounds on the scalar curvature. We show that the weaker condition \eqref{eqn:main-spec} guarantees that the Dirac operator is ``coercive at infinity,'' whereby we can apply the general index theory developed in \cite{anghel:index}.

These ideas are carried out in \S\ref{sec:index-theory}.

\subsection{Spectral scalar curvature of stable minimal surfaces}
To make use of Theorem \ref{thm:main-obstruction}, we need to show that $\Sigma$ has a positive lower bound on its spectral scalar curvature (as in \eqref{eqn:main-spec}).

By the stability inequality and the Gauss equation, any two--sided stable minimal immersion $M^n \to \R^{n+1}$ satisfies
\begin{equation}\label{eqn:nonneg-spec} \int_M |\nabla \phi|^2 + R_M \phi^2 \geq 0 \end{equation}
for all smooth, compactly supported functions $\phi$ on $M$, where $R_M$ is the scalar curvature of $M$ in the pullback metric. The condition \eqref{eqn:nonneg-spec} gives us nonnegative spectral scalar curvature, but not \emph{positive} spectral scalar curvature.

Hence, we need to strengthen the lower bound on spectral scalar curvature. For this purpose, we recall the Hardy inequality for minimal hypersurfaces in $\R^{n+1}$ (see \cite[(1.5)]{Cabre-Miraglio:hardy}):
\begin{equation}\label{eqn:intro-hardy}
\int_M |\nabla \phi|^2 - \frac{(n-2)^2}{4\mathbf{r}^2}\phi^2 \geq 0
\end{equation}
for all smooth, compactly supported functions $\phi$ on $M$, where $\mathbf{r}$ is the Euclidean distance from the origin in $\R^{n+1}$.

Combining \eqref{eqn:nonneg-spec} and \eqref{eqn:intro-hardy}, we obtain
\begin{equation}\label{eqn:intro-spec}
\int_M 2|\nabla \phi|^2 + \left(R_M - \frac{(n-2)^2}{4\mathbf{r}^2}\right)\phi^2 \geq 0
\end{equation}
for any smooth compactly supported $\phi$. We emphasize that the function $h = \frac{(n-2)^2}{4\mathbf{r}^2}$ is identically zero when $n=2$, which is precisely the case of interest. We overcome this issue by considering $\Sigma \times \R$, which is a complete two--sided stable minimal hypersurface in $\R^4$ when $\Sigma$ is a complete two--sided stable minimal surface in $\R^3$. Then by \eqref{eqn:intro-spec}, $\Sigma \times \R$ satisfies the assumption of Theorem \ref{thm:main-obstruction} for $\alpha = 2$ and $h = \frac{1}{4\mathbf{r}^2} > 0$.

These ideas are carried out in \S\ref{sec:hardy}

\subsection{Contracting maps from stable minimal surfaces}
As discussed above, to take advantage of the room provided by the Hardy inequality, our aim is to construct a sufficiently contracting map from $\Sigma \times \R$ to $\mathbb{S}^3$ under the contradiction assumption that $\Sigma$ has faster--than--quadratic area growth. For the purpose of the introduction, we discuss the existence of a sufficiently contracting map from $\Sigma$ to $\mathbb{S}^2$, since this construction contains all of the crucial ideas but avoids more cumbersome details (the $\Sigma \times \R$ construction is carried out in \S\ref{sec:contracting}).

It follows from the Gauss equation that any minimal surface in $\R^3$ has nonpositive curvature. Since two--sided stability passes to covering spaces, we can assume that our complete two--sided stable minimal surface is simply connected. Hence, $\Sigma$ is a two--dimensional Cartan--Hadamard manifold. From this fact, under the contradiction hypothesis that the surface has faster--than--quadratic area growth, we know that the $r$--level set of the intrinsic distance has length at least $\tilde{C}_2\cdot r$ for any $r$ sufficiently large. We can now construct a degree one map from a large annulus in $\Sigma$ to $\mathbb{S}^2$ that is $\frac{C_2}{4r^2}$--contracting on 2--forms by mapping the $r$--level sets of instrinsic distance onto the $z$--level set circles in $\mathbb{S}^2$ (see Figure \ref{fig:intro}). Since the extrinsic Euclidean distance is at most the intrinsic distance, this map is $\frac{C_2}{4\mathbf{r}^2}$--contracting on 2--forms, as required to apply Theorem \ref{thm:main-obstruction} with $h = \frac{1}{4\mathbf{r}^2}$.

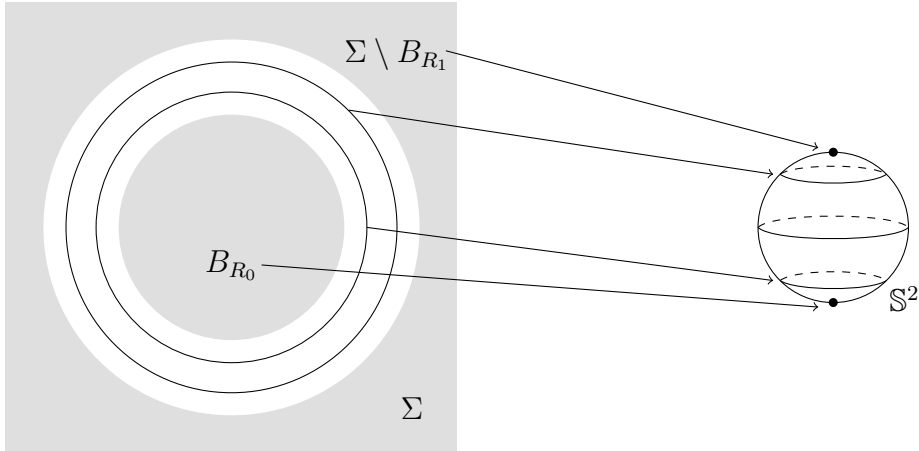
\begin{figure}
\begin{tikzpicture}
	\fill[gray!25!white] (-7, 3) -- (-1, 3) -- (-1, -3) -- (-7, -3);
	\fill[white] (-4, 0) circle (2.5);
	\fill[gray!25!white] (-4, 0) circle (1.5);
	\draw (-4, 0) circle (2.2);
	\draw (-4, 0) circle (1.8);
	\draw (-1.6, -2.4) node {$\Sigma$};
	\draw (-4, -0.5) node {$B_{R_0}$};
	\draw (-1.8, 2.3) node {$\Sigma \setminus B_{R_1}$};

	\draw (3, 0) .. controls (3.2, -0.2) and (4.8, -0.2) .. (5, 0);
	\draw[dashed] (3, 0) .. controls (3.2, 0.2) and (4.8, 0.2) .. (5, 0);
	\draw (3.293, 0.707) .. controls (3.5, 0.55) and (4.5, 0.55) .. (4.707, 0.707);
	\draw[dashed] (3.293, 0.707) .. controls (3.5, 0.85) and (4.5, 0.85) .. (4.707, 0.707);
	\draw (3.293, -0.707) .. controls (3.5, -0.85) and (4.5, -0.85) .. (4.707, -0.707);
	\draw[dashed] (3.293, -0.707) .. controls (3.5, -0.55) and (4.5, -0.55) .. (4.707, -0.707);
	\draw (4, 0) circle (1);
	\fill[black] (4, 1) circle (0.06);
	\fill[black] (4, -1) circle (0.06);
	\draw (4.95, -0.95) node {$\mathbb{S}^2$};
	
	\draw[->] (-1.15, 2.35) -- (3.8, 1.06);
	\draw[->] (-3.6, -0.5) -- (3.8, -1.06);
	\draw[->] (-2.44, 1.56) -- (3.2, 0.707);
	\draw[->] (-2.2, 0) -- (3.2, -0.707);
\end{tikzpicture}
\caption{Degree one contracting map from $\Sigma$ to $\mathbb{S}^2$.}
\label{fig:intro}
\end{figure}

\begin{remark}
Instead of working with a quadratically--decaying positive lower bound on spectral scalar curvature using the Hardy inequality, it is likely possible to give a similar proof in the case of a uniformly positive lower bound on spectral scalar curvature using the conformal change of \cite{Gulliver-Lawson:conformal}. Since the interplay between intrinsic and extrinsic distance becomes more complicated after this conformal change, we find our approach to be more convenient.
\end{remark}

\begin{acknowledgements}
The author is grateful to Otis Chodosh, Shunichiro Orikasa, and Thomas Tony for valuable feedback on an earlier draft, and in particular to Shunichiro Orikasa for pointing out the reference \cite{anghel:index}. The author is also grateful to Paul Minter and Lorenzo Sarnataro for insightful discussions about this work. The author was supported by the NSF grant DMS-2503279.
\end{acknowledgements}

\section{Spectral scalar curvature notation}

Let $(M, g)$ be a complete Riemannian manifold, and let $Q : M \to \R$ be a smooth function. We write
\begin{equation}\label{eqn:gen-spec} -\alpha \triangle + Q \geq 0 \end{equation}
to mean that the following inequality holds for any compactly supported smooth function $\phi$:
\[ \int_M \alpha |\nabla \phi|^2 + Q\phi^2 \geq 0. \]
In this paper, we consider the case where $Q = R_M - h$, where $R_M$ is the scalar curvature of $M$ and $h$ is a smooth function. In this case, we say that the manifold has \emph{spectral scalar curvature} bounded below by $h$. To justify the terminology, note that if $\alpha = 0$, then \eqref{eqn:gen-spec} implies that $R_M \geq h$. Since $-\triangle$ is a nonnegative operator, the condition \eqref{eqn:gen-spec} becomes weaker as $\alpha$ gets larger.

\section{Twisted spinors and spectral scalar curvature}\label{sec:index-theory}
We recall the basic setup and topological facts related to twisted spinor bundles on noncompact Riemannian manifolds, following \cite[Chapter IV, \S6]{Lawson-Michelson:spin}, which in turn follows the ideas of \cite{Gromov-Lawson:spin}.

Let $(M^n, g)$ be a complete Riemannian spin manifold of even dimension $n = 2k \geq 2$. Let $\mathbf{S}_{\C} \to M$ denote the complex spinor bundle of $M$ with canonical metric connection.

Let $E \to M$ be a Hermitian vector bundle together with a metric connection that has \emph{compact support}, meaning that the restriction of the bundle $E$ to the complement of a compact subset of $M$ is a trivial bundle with the trivial de Rham connection.

Let $\mathbf{D}_E$ be the Atiyah--Singer Dirac operator acting on $L^2$--sections of the tensor product $\mathbf{S}_\C \otimes E$. By \cite[Chapter II, Theorem 5.7]{Lawson-Michelson:spin}, $\mathbf{D}_E$ is a self-adjoint operator.

By \cite[Chapter II, Theorem 8.17]{Lawson-Michelson:spin}, we have the twisted Lichnerowicz formula
\begin{equation}\label{eqn:lich} \mathbf{D}_E^2 = \nabla^*\nabla + \frac{1}{4}R_M + \mathcal{R}^E, \end{equation}
where $\mathcal{R}^E$, given by \cite[Chapter II, (8.22)]{Lawson-Michelson:spin}, depends linearly on the curvature tensor of $E$. We note that, by assumption, $\mathcal{R}^E$ vanishes outside a compact subset $K \subset M$. Let $c > 0$ denote a positive constant so that $-\mathcal{R}^E \leq c\ \mathrm{Id}$.

Since $M$ is even dimensional, we can use the complex volume element to induce a $\Z_2$--grading on the complex spinor bundle $\mathbf{S}_{\C} = \mathbf{S}_\C^+ \oplus \mathbf{S}_\C^-$ (see \cite[Chapter II, \S6]{Lawson-Michelson:spin}). With respect to the orthgonal decomposition $\mathbf{S}_{\C} \otimes E= (\mathbf{S}_\C^+ \otimes E) \oplus (\mathbf{S}_\C^- \otimes E)$, the Dirac operator decomposes as
\[ \mathbf{D}_E = \left(\begin{matrix}
0 & \mathbf{D}_E^-\\
\mathbf{D}_E^+ & 0
\end{matrix}\right), \]
where $\mathbf{D}_E^{\pm}$ sends sections of $\mathbf{S}_\C^{\pm} \otimes E$ to sections of $\mathbf{S}_\C^{\mp}\otimes E$. Since $\mathbf{D}_E$ is self-adjoint, we have $\mathrm{ker}(\mathbf{D}_E^-) = \mathrm{coker}(\mathbf{D}_E^+)$. In particular, if $\mathrm{ker}(\mathbf{D}_E)$ is finite dimensional, then both $\mathrm{ker}(\mathbf{D}_E^+)$ and $\mathrm{coker}(\mathbf{D}_E^+)$ are finite dimensional and
\[ \mathrm{ind}(\mathbf{D}_E^+) = \mathrm{dim}(\mathrm{ker}(\mathbf{D}_E^+)) - \mathrm{dim}(\mathrm{ker}(\mathbf{D}_E^-)). \]
We note that, in this case, if $\mathbf{D}_E^+$ has nonzero index, then there is a nonzero $L^2$ section $\psi$ of $\mathbf{S}_C \otimes E$ satisfying $\mathbf{D}_E \psi = 0$.

\subsection{Index theory for the Dirac operator on twisted spinor bundles}
In \cite{Gromov-Lawson:spin} (and \cite[Chapter IV, \S6]{Lawson-Michelson:spin}), the index theory is developed under the assumption that $M$ has uniformly positive scalar curvature. In this subsection, we confirm that the index theory for operators that are ``coercive at infinity'' developed in \cite{anghel:index} (see also \cite[Theorem 6$\frac{4}{5}$]{Gromov:index}) applies when $M$ is only assumed to have uniformly positive \emph{spectral} scalar curvature\footnote{In fact, \cite{Lawson-Michelson:spin, Gromov-Lawson:spin} only require $M$ to have uniformly positive scalar curvature \emph{outside a compact set}. The same applies to our setting as well.}.

Before stating the relative index theorem, we point the reader to \cite[Chapter III, Example 11.13]{Lawson-Michelson:spin} for the definition of the total $\hat{\mathbf{A}}$--class of a vector bundle (here we use $\hat{\mathbf{A}}(M) \coloneq \hat{\mathbf{A}}(TM)$), and to \cite[Chapter III, (11.22)]{Lawson-Michelson:spin} for the definition of the Chern character $\mathrm{ch}(E)$ of a complex bundle $E$ (here we use $\widehat{\mathrm{ch}}(E) \coloneq \mathrm{ch}(E) - \mathrm{rank}(E)$).

\begin{theorem}\label{thm:index}
Suppose $M$ satisfies 
\begin{equation}\label{eqn:spec-scalar} -\alpha \triangle_M + R_M - \kappa \geq 0, \end{equation}
for some constants $0 < \alpha < 4$ and $\kappa_0 > 0$ and a function $\kappa \geq \kappa_0$. Then
\[ \mathrm{ind}(\mathbf{D}_{E}^+) = \int_M \widehat{\mathrm{ch}}(E) \wedge \hat{\mathbf{A}}(M). \]
\end{theorem}

To apply the relative index theorem of \cite[Corollary 3.15]{anghel:index}, it suffices by \cite[Theorem 2.1]{anghel:index} to prove the existence of a positive constant $c_1 > 0$ so that
\begin{equation}\label{eqn:coercive}
\|\mathbf{D}_E\psi\|_{L^2(M)} \geq c_1\|\psi\|_{L^2(M)} \ \ \forall\ \ \psi \in W^{1,2}(M, \mathbf{S}_\C \otimes E), \ \ \mathrm{spt}(\psi) \cap K = \emptyset.
\end{equation}

Recalling that $\mathcal{R}^E$ is supported in $K$, the following lemma implies \eqref{eqn:coercive} for $c_1 = \sqrt{\kappa_0}/2 > 0$. For similar computations, see \cite{Bei-Pipoli:spin, Hirsch-Kazaras-Khuri-Zhang:spec, Chai-Pyo-Wan:spec}.
\begin{lemma}\label{lem:estimate}
Suppose $M$ satisfies \eqref{eqn:spec-scalar} for some constants $\alpha > 0$ and $\kappa_0 > 0$ and a function $\kappa \geq \kappa_0$. If $\psi \in W^{1,2}(M, \mathbf{S}_\C \otimes E)$ with $\|\mathbf{D}_E\psi\|_{L^2(M)}^2 < \infty$, then
\[ \|\mathbf{D}_E\psi\|_{L^2(M)}^2 \geq \left(1-\frac{\alpha}{4}\right)\|\nabla \psi\|_{L^2(M)}^2 + \int_M\frac{\kappa}{4}|\psi|^2 + \int_M \langle \mathcal{R}^E(\psi), \psi\rangle. \]
\end{lemma}
\begin{proof}
Let $\phi$ be a smooth compactly supported function on $M$. For any $\sigma > 0$, we use AM--GM to conclude
\begin{align}\label{eqn:am-gm1}
\|\mathbf{D}_E(\phi\psi)\|_{L^2(M)}^2
& = \int_M \phi^2|\mathbf{D}_E\psi|^2 + 2\int_M \langle (\nabla \phi) \cdot \psi, \phi \mathbf{D}_E\psi\rangle + \int_M |\nabla \phi|^2|\psi|^2\\
& \leq (1+\sigma)\int_M \phi^2|\mathbf{D}_E\psi|^2 + (1+\sigma^{-1})\int_M |\nabla \phi|^2|\psi|^2.\notag
\end{align}
By \eqref{eqn:lich} and integration by parts, we have
\begin{align}\label{eqn:int-by-parts}
\|\mathbf{D}_E(\phi\psi)\|_{L^2(M)}^2
& = \int_M \langle \mathbf{D}_E^2(\phi\psi), \phi\psi\rangle\\
& = \int_M |\nabla (\phi \psi)|^2 + \frac{1}{4}\int_M R_M\phi^2|\psi|^2 + \int_M \phi^2\langle \mathcal{R}^E(\psi), \psi\rangle.\notag
\end{align}
By the Kato inequality and \eqref{eqn:spec-scalar}, we have
\begin{equation}\label{eqn:apply-spec-scalar} \int_M |\nabla (\phi\psi)|^2 \geq \int_M |\nabla (|\phi| |\psi|)|^2 \geq \int_M \alpha^{-1}(\kappa - R_M)\phi^2|\psi|^2. \end{equation}
Writing $|\nabla(\phi\psi)|^2$ as $(1 - \frac{\alpha}{4})|\nabla(\phi\psi)|^2 + \frac{\alpha}{4}|\nabla(\phi\psi)|^2$ in the last line of \eqref{eqn:int-by-parts} and applying \eqref{eqn:apply-spec-scalar}, we find
\begin{equation}\label{eqn:pre-am-gm}
\|\mathbf{D}_E(\phi\psi)\|_{L^2(M)}^2
\geq \left(1 - \frac{\alpha}{4}\right)\int_M |\nabla (\phi\psi)|^2 + \int_M \frac{\kappa}{4}\phi^2|\psi|^2 + \int_M \phi^2\langle \mathcal{R}^E(\psi), \psi\rangle.
 \end{equation}
We again use AM--GM to conclude
 \begin{align}\label{eqn:am-gm2}
 |\nabla(\phi \psi)|^2
 & = \phi^2|\nabla \psi|^2 + 2\langle (\nabla \phi) \otimes \psi, \phi \nabla \psi\rangle + |\nabla \phi|^2|\psi|^2\\
 & \geq (1-\sigma)\phi^2|\nabla \psi|^2 - (\sigma^{-1}-1)|\nabla \phi|^2|\psi|^2.\notag
 \end{align}
 Altogether, by \eqref{eqn:am-gm1}, \eqref{eqn:pre-am-gm}, and \eqref{eqn:am-gm2}, we find a constant $c(\sigma, \alpha) > 0$ so that
 \begin{align}\label{eqn:pre-phi-limit}
 c(\sigma, \alpha)\int_M |\nabla \phi|^2|\psi|^2 + (1+\sigma)\int_M\phi^2|\mathbf{D}_E\psi|^2
 & \geq \left(1 - \frac{\alpha}{4}\right)(1-\sigma)\int_M \phi^2|\nabla \psi|^2\notag\\
 & \hspace{0.5cm}  + \int_M \frac{\kappa}{4}\phi^2|\psi|^2 + \int_M \phi^2\langle \mathcal{R}^E(\psi), \psi\rangle.
 \end{align}

Now, let $\phi_i$ be a sequence of compactly supported smooth functions satisfying
\begin{itemize}
	\item $\phi_i$ converges locally uniformly to 1, and
	\item $\lim_{i \to \infty} \|\nabla \phi_i\|_{L^{\infty}(M)} = 0$.
\end{itemize}
For instance, we can take $\phi_i$ to be a smoothing of a radial linear cutoff on the annulus $B_{2i}(x_0) \setminus B_i(x_0)$. Taking the limit of \eqref{eqn:pre-phi-limit} as $i \to \infty$ yields\footnote{We note that it follows from \eqref{eqn:pre-phi-limit} along this sequence that $\sqrt{\kappa}\psi$ is in $L^2$, where we have used the fact that $\psi$ and $\nabla \psi$ are in $L^2$ and that $\mathcal{R}^E$ is continuous and compactly supported.}
\begin{equation}\label{eqn:pre-sigma-limit}
(1+\sigma)\|\mathbf{D}_E\psi\|_{L^2(M)}^2
\geq \left(1 - \frac{\alpha}{4}\right)(1-\sigma)\|\nabla \psi\|_{L^2(M)}^2 + \int_M\frac{\kappa}{4}|\psi|^2 + \int_M \langle \mathcal{R}^E(\psi), \psi\rangle.
\end{equation}
The desired inequality follows from taking the limit of \eqref{eqn:pre-sigma-limit} as $\sigma \to 0$.
\end{proof}

We emphasize that, in addition to enabling the index theory of \cite{anghel:index}, Lemma \ref{lem:estimate} also implies that $\mathbf{D}_{\C^l}$ has zero kernel (i.e.\ when the bundle $E$ is globally trivial). This result is proved by \cite{Bei-Pipoli:spin}.

\begin{proposition}[\cite{Bei-Pipoli:spin}]\label{prop:zero-ker}
Suppose $M$ satisfies \eqref{eqn:spec-scalar} for some constants $0 < \alpha < 4$ and $\kappa_0 > 0$ and a function $\kappa \geq \kappa_0$. Then the operator $\mathbf{D}_{\C^l}$ acting on $L^2$ sections of $\mathbf{S}_\C\otimes \C^l$ has zero kernel.
\end{proposition}
\begin{proof}
Suppose $\psi$ is in the kernel of $\mathbf{D}_{\C^l}$. Since the connection on the trivial bundle is flat, Lemma \ref{lem:estimate} implies
\[ \frac{\kappa_0}{4}\|\psi\|_{L^2(M)}^2 \leq \left(1 - \frac{\alpha}{4}\right)\|\nabla \psi\|_{L^2(M)}^2 + \int_M\frac{\kappa}{4}|\psi|^2 \leq 0, \]
so we have $\psi = 0$.
\end{proof}

\subsection{Obstruction to spectral scalar curvature lower bounds}
We now use Theorem \ref{thm:index} to establish an obstruction to positive spectral scalar curvature (see Theorem \ref{thm:main-obstruction}).

\begin{theorem}\label{thm:obstruction}
There are constants $C_n > 0$ so that the following holds. Suppose $(M^n, g)$ is a complete spin manifold of dimension $n \geq 2$ satisfying
\begin{equation}\label{eqn:spec-obst} -\alpha \triangle_M + R_M - h \geq 0, \end{equation}
for $0 < \alpha < 4$, where $h : M \to \R$ is a smooth positive function. Then there is no Lipschitz map $f: M \to \mathbb{S}^n$ satisfying
\begin{itemize}
	\item $f$ is has nonzero degree,
	\item $f$ is constant outside a compact subset of $M$, and
	\item $f$ is $C_n h$--contracting on 2--forms.
\end{itemize}
\end{theorem}
\begin{proof}
Suppose otherwise for contradiction, so such a map $f$ exists. We carry out the proof when the dimension of $M$ is even (in the odd case, the same argument works by replacing $\mathbb{S}^2$ with $\mathbb{S}^3$).

We note that the assumptions here do not ensure \emph{uniformly} positive spectral scalar curvature, which is required to apply the index theory from Theorem \ref{thm:index}. As in \cite[Proof of Theorem 6.12]{Gromov-Lawson:spin}, we can apply the index theory after first taking a product with a large 2--sphere. Enlarging the compact set $K$ (outside of which $E$ is trivial) if necessary, we suppose that $f$ is constant outside $K$. Let $h_0 = \min_{K}h > 0$. Define the constant
\[ \kappa_0 \coloneq 2\min\left\{\frac{C_nh_0}{\mathrm{Lip}\ f},\ \sqrt{C_nh_0} \right\} > 0. \]
Consider the $(n+2)$-dimensional spin manifold
\[ \tilde{M} \coloneq \left(M \times S^2,\ g \oplus \frac{2}{\kappa_0} g_{\mathbb{S}^2}\right). \]
Then $\tilde{M}$ satisfies
\begin{equation}\label{eqn:new-spec} -\alpha \triangle_{\tilde{M}} + R_{\tilde{M}} - (\kappa_0 + \tilde{h}) \geq 0, \end{equation}
where $\tilde{h}$ is the $S^2$--invariant extension of $h$ to $\tilde{M}$. Indeed, by \cite[Theorem 1]{Fischer-Colbrie-Schoen:stable}, there is a smooth positive function $u : M \to \R_{>0}$ satisfying
\[ -\alpha \triangle_M u + (R_M - h)u = 0. \]
If we let $\tilde{u}$ denote the $S^2$--invariant extension of $u$ to $\tilde{M}$, and $\tilde{R}_M$ ($= R_{\tilde{M}} - \kappa_0)$ the $S^2$--invariant extension of the scalar curvature function of $M$ to $\tilde{M}$, then we have
\[ 0 = -\alpha \triangle_{\tilde{M}} \tilde{u} + (\tilde{R}_M - \tilde{h})\tilde{u} = -\alpha \triangle_{\tilde{M}} \tilde{u} + (R_{\tilde{M}} - \kappa_0 - \tilde{h})\tilde{u}. \]
By \cite[Theorem 1]{Fischer-Colbrie-Schoen:stable}, \eqref{eqn:new-spec} holds. Let $\kappa \coloneq \kappa_0 + \tilde{h} \geq \kappa_0$. In particular, the assumptions of Theorem \ref{thm:index} are satisfied for $\tilde{M}$.

By the choice of the constant $\kappa_0$, the map $f \times \frac{\kappa_0}{2} : \tilde{M} \to \mathbb{S}^n \times \mathbb{S}^2$ is Lipschitz, has nonzero degree, and is $C_n\kappa$--contracting on 2--forms. Furthermore, $f \times \frac{\kappa_0}{2}$ maps the complement of a compact subset of $\tilde{M}$ into $\{p_0\} \times \mathbb{S}^2 \subset \mathbb{S}^n \times \mathbb{S}^2$. Let $s_n : \mathbb{S}^n \times \mathbb{S}^2 \to \mathbb{S}^{n+2}$ be a fixed Lipschitz degree one map that is constant on $\{p_0\} \times \mathbb{S}^2$ (this is a ``smashing map,'' see \cite[Proof of Theorem 6.12]{Gromov-Lawson:spin}). Then
\[ \tilde{f} \coloneq s_n \circ \left(f \times \frac{\kappa_0}{2}\right) : \tilde{M} \to \mathbb{S}^{n+2} \]
is Lipschitz, has nonzero degree, is constant outside a compact set, and is $(\mathrm{Lip}\ s_n)^2C_n\kappa$--contracting on 2--forms. Let $\tilde{C}_n \coloneq (\mathrm{Lip}\ s_n)^2C_n$.

For ease of notation, let $m \coloneq n + 2$. Let $E_0 \to \mathbb{S}^m$ be a complex vector bundle with nonzero top Chern class $c_{m/2}(E_0) \neq 0$ (such a bundle exists, see \cite[Proof of Theorem 5.5]{Lawson-Michelson:spin} and \cite{Atiyah-Hirzebruch}). Equip $E_0$ with a fixed Hermitian metric and a fixed metric connection. Using the map $\tilde{f}$, we let $E \coloneq \tilde{f}^*E_0$. Since $\tilde{f}$ is constant outside a compact set, $E$ is compactly supported.

We compute the index of $\mathbf{D}_E^+$ as constructed above using Theorem \ref{thm:index}. Since $\tilde{f}$ has nonzero degree and the Chern classes are functorial, we have
\[ c_j(E) = \begin{cases}
1 & j = 0\\
0 & 1 \leq j \leq m/2-1\\
\omega \neq 0 & j = m/2.
\end{cases} \]
Hence, we compute the reduced Chern character to be
\[ \widehat{\mathrm{ch}}(E) = (-1)^{m/2-1}\frac{1}{(m/2-1)!}\omega \neq 0. \]
Since $\omega$ has top degree, the right hand side of Theorem \ref{thm:index} is
\[ \int_{\tilde{M}} \widehat{\mathrm{ch}}(E) \wedge \hat{\mathbf{A}}(\tilde{M}) = (-1)^{m/2-1}\frac{1}{(m/2-1)!}\int_{\tilde{M}} \omega \neq 0. \]
Then by Theorem \ref{thm:index}, $\mathrm{ker}(\mathbf{D}_E) \neq 0$.

By the contracting property of $\tilde{f}$, and the fact that $E_0$ is a fixed bundle in each dimension, there is a fixed constant $K_n$ depending only on $n$ so that we have
\[ -\mathcal{R}_E \leq K_n\tilde{C}_n\kappa\ \mathrm{Id}. \]
If we take $\tilde{C}_n$ to be $\frac{1}{8K_n}$ (i.e.\ take $C_n$ to be $\frac{1}{8K_n(\mathrm{Lip}\ s_n)^2}$), then we have
\[ -\mathcal{R}_E \leq \frac{1}{8}\kappa\ \mathrm{Id}. \]
Suppose $\psi$ is an $L^2$ section of $\mathbf{S}_\C \otimes E$ in the kernel of $\mathbf{D}_E$. By Lemma \ref{lem:estimate}, we have
\[ \int_M \frac{\kappa}{4}|\psi|^2 \leq \int_M \frac{\kappa}{8}|\psi|^2. \]
Since $\kappa \geq \kappa_0 > 0$, we have $\psi = 0$. This conclusion contradicts the fact that $\mathrm{ker}(\mathbf{D}_E) \neq 0$.
\end{proof}

\begin{remark}
The range of permissible values of the coefficient $\alpha$ can actually be slightly increased using the improved Kato inequality (see \cite[Lemma 3.1]{Feehan:Kato}). In this situation we don't need to make essential use of this extra room, unlike in the case of harmonic 1--forms, because the Lichnerowicz formula for spinors has a good factor of $\frac{1}{4}$ in front of the scalar curvature term.
\end{remark}

\section{Stable minimal surfaces}
We apply the Dirac index theory developed in the previous section to give a new proof of the classification of complete two--sided stable minimal surfaces in $\R^3$.

\subsection{Spectral scalar curvature of stable minimal surfaces}\label{sec:hardy}
We prove a positive--but--decaying lower bound for the spectral scalar curvature of a two--sided stable minimal surface in $\R^3$, taking advantage of the Hardy inequality for minimal hypersurfaces (see \cite{Cabre-Miraglio:hardy}). Since the Hardy inequality is only effective for minimal hypersurfaces in dimension 4 and higher, we take a product with $\R$.
\begin{lemma}\label{lem:stable-spec}
Let $\Sigma \to \R^3$ be a two--sided stable minimal immersion, and let $g$ denote the pullback metric on $\Sigma$. Consider the product manifold
\[ (M, \tilde{g}) \coloneq (\Sigma \times \R, g \oplus dt^2). \]
Let $\mathbf{r}$ be the restriction of the Euclidean distance in $\R^4$ from $\vec{0}$ to the $\R$--invariant immersion $M \to \R^4$. For any $\beta > 0$, we have
\begin{equation}\label{eqn:stable-spec}
-(1 + \beta)\triangle_{M} + R_M - \frac{\beta}{4\mathbf{r}^2} \geq 0.
\end{equation}
\end{lemma}
\begin{proof}
We note that $\tilde{g}$ is the pullback metric on $M$ for the $\R$--invariant immersion $M \to \R^4$.
Moreover, the immersion $M \to \R^4$ is also a two--sided stable minimal immersion. Indeed, two--sided stability is equivalent to the inequality
\[ -\triangle_{\Sigma} - |A_{\Sigma}|^2 \geq 0. \]
By \cite[Theorem 1]{Fischer-Colbrie-Schoen:stable}, there is a positive function $u : \Sigma \to \R_{>0}$ so that
\[ -\triangle_{\Sigma}u - |A_{\Sigma}|^2u = 0. \]
Let $\tilde{u}$ denote the $\R$--invariant extension of $u$ to $\Sigma \times \R$. Since $|A_M|^2 = |A_{\Sigma}|^2$, the positive function $\tilde{u}$ satisfies
\[ -\triangle_{M}\tilde{u} - |A_{M}|^2\tilde{u} = 0. \]
Then by \cite[Theorem 1]{Fischer-Colbrie-Schoen:stable}, $M$ is stable.

By the Gauss equation and the minimality of $M$, we know that $M$ has nonnegative spectral scalar curvature:
\begin{equation}\label{eqn:spec-nonneg} -\triangle_M + R_M \geq 0. \end{equation}
By the Hardy inequality for minimal hypersurfaces in $\R^4$ (see \cite[(1.5)]{Cabre-Miraglio:hardy}), we have
\begin{equation}\label{eqn:hardy} -\triangle_M - \frac{1}{4\mathbf{r}^2} \geq 0. \end{equation}
Adding \eqref{eqn:spec-nonneg} to $\beta$ times \eqref{eqn:hardy} gives \eqref{eqn:stable-spec}.
\end{proof}

\subsection{Construction of contracting map}\label{sec:contracting}
Here we aim to construct a degree one map from $\Sigma \times \R$ to $\mathbb{S}^3$ that is sufficiently contracting on 2--forms.

We begin with two preliminary propositions about Cartan--Hadamard manifolds. These results are standard in Riemannian geometry, but we include the details of the precise formulations required for our construction.

The first preliminary proposition relates to the gradient of the radial projection maps.

\begin{proposition}\label{prop:cartan-hadamard1}
Let $(M^n, g)$ be a complete simply connected manifold with nonpositive sectional curvature, and fix $x_0 \in M$. Then $\exp_{x_0} : \R^n \to M$ is a diffeomorphism. For $0 < \rho < R$, the radial projection map
\[ \pi_{R, \rho} : \partial B_R(x_0) \to \partial B_\rho(x_0),\ \ \pi_{R, \rho}(x) \coloneq \exp_{x_0}\left(\frac{\rho}{R}\exp_{x_0}^{-1}(x)\right) \]
is smooth, depends smoothly on $R$ and $\rho$, and satisfies
\[ |\nabla \pi_{R,\rho}| \leq \frac{\rho}{R}. \]
\end{proposition}
\begin{proof}
The fact that $\exp_{x_0}$ is a diffeomorphism is the classical theorem of Cartan and Hadamard (see for instance \cite[Theorem 6.2.2]{Petersen:geometry}. The fact that $\pi_{R, \rho}$ is smooth and depends smoothly on $R$ and $\rho$ follows from the fact that $\exp_{x_0}$ is a diffeomorphism and dilation in $\R^n$ by $\frac{\rho}{R}$ is a smooth map depending smoothly on $R$ and $\rho$.

For the gradient bound, let us write the metric $g$ in geodesic normal coordinates centered at $x_0$. Then $g = dr^2 + g_r$. By \cite[Proposition 3.2.11]{Petersen:geometry}, we have
\[ \frac{\partial}{\partial r}g_r = L_{\partial r} g = 2\mathrm{Hess}\ r. \]
By the Rauch comparison in nonpositive curvature (see \cite[Theorem 6.4.3]{Petersen:geometry}), we have
\[ \mathrm{Hess}\ r \geq \frac{1}{r}g_r. \]
Let $v$ be a unit vector tangent to $\mathbb{S}^{n-1}$. In these polar coordinates, we have along radial geodesics
\[ \log\left(\frac{\sqrt{g_R(v,v)}}{\sqrt{g_\rho(v,v)}}\right) = \frac{1}{2}\log\left(\frac{g_R(v, v)}{g_\rho(v,v)}\right) \geq \int_{\rho}^R \frac{1}{r}dr = \log(R/\rho). \]
In other words (recalling that $\sqrt{g(v,v)}$ is the length of $v$),
\[ g_R \geq \frac{R}{\rho}g_\rho. \]
Since $\pi_{R, \rho}$ is the identity on the spherical part of the polar coordinates, we have the desired gradient estimate.
\end{proof}

The second preliminary proposition establishes a monotonicity formula for the area of geodesic balls.

\begin{proposition}\label{prop:cartan-hadamard2}
Let $(\Sigma^2, g)$ be a complete simply connected surface with nonpositive curvature, and fix $x_0 \in \Sigma$. Let
\[ a(r) \coloneq \mathrm{Area}(B_r(x_0)), \ \ l(r) \coloneq \mathrm{Length}(\partial B_r(x_0)). \]
Then
\[ \left(\frac{a(r)}{\pi r^2}\right)' \geq 0\ \ \mathrm{and}\ \ l(r) \geq \frac{2a(r)}{r}. \]
\end{proposition}
\begin{proof}
The second inequality follows directly from the first inequality. For the first inequality, working in geodesics polar coordinates centered at $x_0$, we have
\[ \mathrm{area} = \lambda(r, \theta)dr \wedge d\theta. \]
Recall by \cite[\S 2.1.3]{Petersen:geometry} that we have
\[ (\partial_r \lambda) dr\wedge d\theta = L_{\partial_r}\mathrm{area} = (\triangle r)\mathrm{area} = (\triangle r)\lambda dr\wedge d\theta. \]
By the Rauch comparison (see \cite[Theorem 6.4.3]{Petersen:geometry}), we have
\[ \triangle r \geq r^{-1}. \]
Hence,
\[ \partial_r\lambda \geq r^{-1}\lambda. \]
We compute
\begin{align*}
	\frac{d}{dr}\left(\frac{a(r)}{\pi r^2}\right)
	& = \frac{d}{dr}\left(\frac{\int_0^r \int_0^{2\pi} \lambda(\rho, \theta) d\rho \wedge d\theta}{\pi r^2}\right)\\
	& = \frac{\pi r^2\int_0^{2\pi}\lambda(r, \theta) d\theta - 2\pi r\int_0^r\int_0^{2\pi} \lambda(\rho, \theta) d\rho \wedge d\theta}{(\pi r^2)^2}\\
	& = \frac{\int_0^r\left[(2\pi \rho \int_0^{2\pi} \lambda(r,\theta)d\theta - 2\pi r \int_0^{2\pi} \lambda(\rho, \theta)d\theta\right]d\rho }{(\pi r^2)^2}.
\end{align*}
To show that this derivative is nonnegative, it suffices to check that
\[ r \mapsto \frac{\int_0^{2\pi}\lambda(r,\theta)d\theta}{2\pi r} \]
is nondecreasing (since the limit as $r \to 0$ is 1). We compute
\begin{align*}
	\frac{d}{dr}\left(\frac{\int_0^{2\pi}\lambda(r,\theta)d\theta}{2\pi r}\right)
	& = \frac{2\pi r \int_0^{2\pi} \partial_r\lambda d\theta - 2\pi\int_0^{2\pi} \lambda d\theta}{(2\pi r)^2}\\
	& \geq \frac{2\pi r \int_0^{2\pi} r^{-1}\lambda d\theta - 2\pi\int_0^{2\pi} \lambda d\theta}{(2\pi r)^2} = 0,
\end{align*}
which concludes the proof.
\end{proof}

We now proceed to the construction of a contracting map.

\begin{lemma}\label{lem:construction}
Let $\Sigma$ be a noncompact, simply connected surface. Let $\Sigma \to \R^3$ be a complete minimal immersion, and let $g$ denote the pullback of the Euclidean metric. Fix $x_0 \in \Sigma$. There is a constant $C > 0$ so that for any $\eps > 0$, there is a Lipschitz map $f : (\Sigma \times \R, g \oplus dt^2) \to \mathbb{S}^3$ satisfying
\begin{itemize}
	\item $f$ has degree one,
	\item $f$ is constant outside a compact set, and
	\item $f$ is $C \cdot \max\{A + \eps, (A+\eps)^2\}/r(x, t)^2$--contracting on 2--forms, where
	\[ r(x,t) \coloneq d_{g \oplus dt^2}((x,t), (x_0, 0)) \]
	and
	\[ A \coloneq \lim_{\rho \to \infty} \frac{\pi \rho^2}{\mathrm{Area}_g(B_\rho(x_0) \subset \Sigma)}. \]
\end{itemize}
\end{lemma}

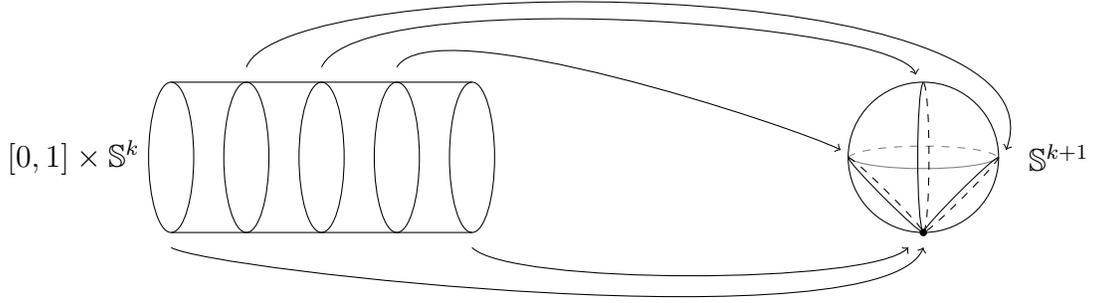
\begin{figure}
\begin{tikzpicture}
	\draw (-6, 1) -- (-2, 1);
	\draw (-6, -1) -- (-2, -1);
	\draw (-6, 0) ellipse (0.3 and 1);
	\draw (-5, 0) ellipse (0.3 and 1);
	\draw (-4, 0) ellipse (0.3 and 1);
	\draw (-3, 0) ellipse (0.3 and 1);
	\draw (-2, 0) ellipse (0.3 and 1);
	\draw (-7.3, 0) node {$[0, 1] \times \mathbb{S}^k$};
	\draw (5.8, 0) node {$\mathbb{S}^{k+1}$};
	
	\draw[gray] (3, 0) .. controls (3.2, -0.2) and (4.8, -0.2) .. (5, 0);
	\draw[gray, dashed] (3, 0) .. controls (3.2, 0.2) and (4.8, 0.2) .. (5, 0);
	\draw (4, -1) .. controls (3.9, -0.9) and (3.9, 0.9) .. (4, 1);
	\draw[dashed] (4, -1) .. controls (4.1, -0.9) and (4.1, 0.9) .. (4, 1);
	\draw (3, 0) .. controls (3, -0.08) and (3.92, -1) .. (4, -1);
	\draw[dashed] (3, 0) .. controls (3.08, 0) and (4, -0.92) .. (4, -1);
	\draw (5, 0) .. controls (4.92, 0) and (4, -0.92) .. (4, -1);
	\draw[dashed] (5, 0) .. controls (5, -0.08) and (4.08, -1) .. (4, -1);
	\draw (4, 0) circle (1);
	\fill[black] (4, -1) circle (0.05);
	
	\draw[->] (-6, -1.2) .. controls (-5.5, -1.5) and (3.5, -2.5) .. (4, -1.2);
	\draw[->] (-5, 1.2) .. controls (-4.5, 2.5) and (6, 2.5) .. (5.1, 0.1);
	\draw[->] (-4, 1.2) .. controls (-3.5, 2.3) and (3.7, 1.8) .. (3.9, 1.1);
	\draw[->] (-3, 1.2) .. controls (-2.5, 2) and (2.3, 0.4) .. (2.9, 0.1);
	\draw[->] (-2, -1.2) .. controls (-1.5, -1.7) and (3.3, -1.7) .. (3.8, -1.2);
\end{tikzpicture}
\caption{A ``smashing'' map from $[0, 1] \times \mathbb{S}^k$ to $\mathbb{S}^{k+1}$.}
\label{fig:smash}
\end{figure}

\begin{proof}
First, we set up some notation. By the Gauss equation and the minimality of $\Sigma$, the pullback metric $g$ has nonpositive curvature. Then by Proposition \ref{prop:cartan-hadamard2}, the function $\rho \mapsto \frac{\pi \rho^2}{a(\rho)}$ is nonincreasing, so the limit $A$ is well-defined and finite. Choose $D > 0$ so that $\frac{\pi D^2}{a(D)} \leq A + \eps$. By Proposition \ref{prop:cartan-hadamard2}, we have $l(D) \geq \frac{2\pi D}{A+\eps}$. We set $R = \sqrt{2}D$.

{\bf Step 0.} \emph{Map from $\partial B_r^{\Sigma}$ to $\mathbb{S}^1$.} Let $f_0 : \partial B_D^{\Sigma}(x_0) \to \mathbb{S}^1$ be a degree one smooth map with $|\nabla f_0| \leq \frac{A + \eps}{D}$ (i.e.\ $f_0$ is the inverse of a constant speed parametrization of $\partial B_D^{\Sigma}(x_0)$ by $\mathbb{S}^1$).

{\bf Step 1.} \emph{Map from $\partial B_r^{\Sigma \times \R}$ to $\mathbb{S}^2$.} We first handle the ``collar'' of $\partial B_r^{\Sigma \times \R}$ that is far from $x_0$ in the $\Sigma$ factor of the product. Note that
\[ \partial B_R^{\Sigma \times \R}(x_0, 0) = \bigcup_{t \in [-R, R]} \partial \overline{B}_{\sqrt{R^2 - t^2}}^{\Sigma}(x_0)\times \{t\}. \]
We define a Lipschitz map
\[ f_1^{\mathrm{collar}} : \partial B_R^{\Sigma \times \R}(x_0, 0) \cap \{t \in [-D, D]\} \to [-1, 1] \times \mathbb{S}^1 \]
by
\[
f_1^{\mathrm{collar}}(x, t) = \{t/D\} \times f_0(\pi_{\sqrt{R^2 - t^2}, D}^{\Sigma}(x)),
\]
where $\pi_{\sqrt{R^2 - t^2}, D}^{\Sigma}$ is the radial projection map in $\Sigma$ defined in Proposition \ref{prop:cartan-hadamard1}. By construction and Proposition \ref{prop:cartan-hadamard1}, there is an orthonormal basis $\{e_1, e_2\}$ at every point so that
\[ |\nabla_{e_1}f_1^{\mathrm{collar}}| \leq \frac{1}{D},\ \ |\nabla_{e_2}f_1^{\mathrm{collar}}| \leq \frac{A+\eps}{D}. \]

To handle the complement of the ``collar'' of $\partial B_r^{\Sigma \times \R}$, we require a smashing map. Let $\mathrm{sm}_1 : [-1, 1] \times \mathbb{S}^1 \to \mathbb{S}^2$ be a degree one Lipschitz map with the property that
\[ \mathrm{sm}_1([-1, 1] \times \{(0, -1) \in \mathbb{S}^1\}) = \mathrm{sm}_1(\{-1\} \times \mathbb{S}^1) = \mathrm{sm}_1(\{1\} \times \mathbb{S}^1) = (0, 0, -1) \in \mathbb{S}^2. \]
($\mathrm{sm}_1$ is a ``smashing map,'' see Figure \ref{fig:smash} and \cite[Proof of Theorem 6.12]{Gromov-Lawson:spin}.) We can arrange this map to have $|\nabla \mathrm{sm}_1| \leq \pi$.

Finally, we can construct our map from $\partial B_r^{\Sigma \times \R}$ to $\mathbb{S}^2$. Let $f_1 : \partial B_R^{\Sigma \times \R}(x_0, 0) \to \mathbb{S}^2$ be the degree one Lipschitz map given by
\[ f_1(x,t) = \begin{cases}
(0, 0, -1) & |t| \geq D\\
\mathrm{sm}_1(f_1^{\mathrm{collar}}(x,t)) & |t| < D.
\end{cases} \]
Then there is an orthonormal basis $\{e_1, e_2\}$ at every point so that 
\[ |\nabla_{e_1}f_1| \leq \frac{\pi}{D},\ \ |\nabla_{e_2}f_1| \leq \frac{\pi(A+\eps)}{D}. \]

{\bf Step 2.} \emph{Map from $\Sigma \times \R$ to $\mathbb{S}^3$.} We first handle a large annulus in $\Sigma \times \R$. Let
\[ f_2^{\mathrm{annulus}}: \{R \leq r(x,t) \leq Re^{1/(A+\eps)}\} \subset \Sigma \times \R \to [0, 1] \times \mathbb{S}^2 \]
be the Lipschitz map given by
\[ f_2^{\mathrm{annulus}}(x,t) = \left\{(A+\eps)\log\left(\frac{r(x,t)}{R}\right)\right\} \times f_1(\pi_{r(x,t), R}^{\Sigma \times \R}(x)), \]
where $\pi_{r(x,t), R}^{\Sigma \times \R}$ is the radial projection map in $\Sigma \times \R$ from Proposition \ref{prop:cartan-hadamard1}. By construction and Proposition \ref{prop:cartan-hadamard1} (using the fact that $\Sigma \times \R$ is also complete, simply connected, and has nonpositive curvature), there is an orthonormal basis $\{e_1, e_2, e_3\}$ at every point so that
\begin{align*}
|\nabla_{e_1}f_2^{\mathrm{annulus}}| & \leq \frac{\pi}{D} \cdot\frac{R}{r(x,t)} = \frac{\sqrt{2}\pi}{r(x,t)},\\
|\nabla_{e_2}f_2^{\mathrm{annulus}}| & \leq \frac{\pi(A+\eps)}{D}\cdot \frac{R}{r(x,t)} = \frac{\sqrt{2}\pi(A+\eps)}{r(x,t)},\\
|\nabla_{e_3}f_2^{\mathrm{annulus}}| & \leq \frac{A+\eps}{r(x,t)}.
\end{align*}

To handle the complement of the annulus, we require a smashing map. Let $\mathrm{sm}_2: [0, 1] \times \mathbb{S}^2 \to \mathbb{S}^3$ be a degree one Lipschitz map with the property that
\[ \mathrm{sm}_2([0, 1] \times \{(0,0, -1)\}) = \mathrm{sm}_2(\{0\} \times \mathbb{S}^2) = \mathrm{sm}_2(\{1\} \times \mathbb{S}^2) = (0, 0, 0, -1) \in \mathbb{S}^3. \]
($\mathrm{sm}_2$ is a ``smashing map,'' see Figure \ref{fig:smash} and \cite[Proof of Theorem 6.12]{Gromov-Lawson:spin}.) We can arrange this map to have $|\nabla \mathrm{sm}_2| \leq 2\pi$.

At last, we define $f : \Sigma \times \R \to \mathbb{S}^3$ by
\[ f(x, t) = \begin{cases}
(0, 0, 0, -1) & r(x,t) \leq R\\
\mathrm{sm}_2(f_2^{\mathrm{annulus}}(x,t)) & r(x,t) \in (R, Re^{1/(A+\eps)})\\
(0, 0, 0, -1)& r(x,t) \geq Re^{\pi/(A + \eps)}.
\end{cases} \]
Then $f$ is a degree one Lipschitz map, constant outside a compact set, and there is an orthonormal basis $\{e_1, e_2, e_3\}$ at every point so that
\[ |\nabla_{e_1}f| \leq \frac{2\sqrt{2}\pi^2}{r(x,t)}, \ \ |\nabla_{e_2}f| \leq \frac{2\sqrt{2}\pi^2(A+\eps)}{r(x,t)}, \ \ |\nabla_{e_3}f| \leq \frac{2\pi(A+\eps)}{r(x,t)}. \]
Hence, $f$ is $(8\pi^4)\cdot \max\{A+\eps, (A+\eps)^2\}/r(x,t)^2$--contracting on 2--forms.
\end{proof}

\subsection{Proof of Theorem \ref{thm:main}}
Let $\Sigma^2 \to \R^3$ be a complete two--sided stable minimal immersion. Since two--sided stability passes to covering space (by lifting the positive solution of the Jacobi operator using \cite[Theorem 1]{Fischer-Colbrie-Schoen:stable}), it suffices to assume that $\Sigma$ is simply connected.

Fix $x_0 \in \Sigma$ (and assume that it maps to $\vec{0} \in \R^3$ under the minimal immersion). Suppose for contradiction that $\Sigma$ satisfies
\[ \lim_{\rho \to \infty} \frac{\mathrm{Area}(B_\rho(x_0))}{\pi \rho^2} = +\infty. \]
Then by Lemma \ref{lem:construction}, for any $\delta > 0$, there is a Lipschitz map $f : \Sigma \times \R \to \mathbb{S}^3$ satisfying
\begin{itemize}
	\item $f$ has degree one,
	\item $f$ is constant outside a compact set, and
	\item $f$ is $\delta/r(x,t)^2$--contracting on 2--forms.
\end{itemize}
Moreover, by Lemma \ref{lem:stable-spec}, taking $\beta = 1$ for instance, $\Sigma \times \R$ satisfies
\[ -2\triangle_{\Sigma \times \R} + R_{\Sigma \times \R} - \frac{1}{4r(x,t)^2} \geq 0, \]
where we used the fact that the extrinsic distance from $\vec{0}$ is at most the intrinsic distance from $(x_0, 0).$\footnote{To avoid the singularity at $r(x,t) = 0$, we can take a smoothing of $\max\{C, \frac{1}{4r(x,t)^2}\}$, which presents no problem because the map $f$ constructed in Lemma \ref{lem:construction} is constant inside a geodesic ball of fixed radius.} Since every oriented surface is spin,\footnote{In higher dimensions, it is even true that every two--sided codimension one immersion in Euclidean space is spin (see \cite[Proposition 1.4.1]{Ginoux:spin}).} these conclusions contradict Theorem \ref{thm:obstruction}. It now follows from the log--cutoff (see for instance \cite[Theorem 2.10]{Colding-Minicozzi:minimal}) that $\Sigma$ is flat.

In fact, this argument gives some explicit, uniform, a priori constant $C > 0$ so that any simply connected complete two--sided stable minimal immersion in $\R^3$ satisfies
\[ \mathrm{Area}(B_\rho(x_0)) \leq C \rho^2. \]

 \bibliographystyle{amsalpha}
 \bibliography{stable_and_spin}

\end{document}